\documentclass[runningheads,11pt]{llncs}

\usepackage{amssymb}
\usepackage{graphicx}
\usepackage{amsmath}
\usepackage{stmaryrd}
\usepackage{url}
\usepackage[color,matrix,arrow]{xy}
\usepackage{wrapfig}
\usepackage{textcomp}
\usepackage{bussproofs}
\usepackage{proof}
\usepackage{color}
\usepackage{enumerate}

\newcommand{\mf}{\mathfrak}
\newcommand{\mb}{\mathbb}
\newcommand{\mc}{\mathcal}

\newcommand{\imp}{\rightarrow}
\newcommand{\bimp}{\leftrightarrow}
\newcommand{\Imp}{\Rightarrow}
\newcommand{\sub}{\subseteq}

\newtheorem{fact}[theorem]{Fact}

\title{Residuated Basic Logic I}
\author{Minghui Ma\inst{1} \and Zhe Lin \inst{2}}
\institute{Institute for Logic and Intelligence, Southwest University,\\
Beibei District, Chongqing, 400715, China.\\
\email{mmh.thu@gmail.com}
\and Corresponding Author.\\
Institute of Logic and Cognition, Sun Yat-sen University\\
No. 135, Xingang Xi Road, Guangzhou, China\\
Faculty of Mathematics and Computer Science, Adam Mickiewicz University,
Umultowska 87, 61-614 Pozna\'{n}, Poland\\
\email{pennyshaq@gmail.com}}

\begin{document}

\maketitle

\begin{abstract}
We study the residuated basic logic ($\mathsf{RBL}$) of residuated basic algebra in which the basic implication of Visser's basic propositional logic ($\mathsf{BPL}$) is interpreted as the right residual of a non-associative binary operator $\cdot$ (product). We develop an algebraic system $\mathsf{S_{RBL}}$ of residuated basic algebra by which we show that $\mathsf{RBL}$ is a conservative extension of $\mathsf{BPL}$. We present the sequent formalization $\mathsf{L_{RBL}}$ of $\mathsf{S_{RBL}}$ which is an extension of distributive  full non-associative Lambek calculus ($\mathsf{DFNL}$), and show
that the cut elimination and subformula property hold for it.
\end{abstract}


\section{Introduction and Preliminaries}

Basic propositional logic ($\mathsf{BPL}$) was introduced by Visser \cite{visser81} as a subintuitionistic logic, and recently developed by many other authors (\cite{AR98}, \cite{AR01}, \cite{SO97}, \cite{IKK01}, \cite{SWZ98}). The aim of this paper is to investigate the relationship between $\mathsf{BPL}$ and substructural logics. Inspired from Buszkowski's idea in \cite{Bus11}, we treat the basic implication $\rightarrow$ in $\mathsf{BPL}$ as the right residual of a designated binary operator $\cdot$, in addition, admitting some structural rules (weakening and restricted contraction). Thus $\mathsf{BPL}$ can be extended conservatively to a substructural logic.

The language of basic propositional logic is extended by adding the binary operator $\cdot$ and its left residual $\leftarrow$. Then we introduce the residuated basic logic ($\mathsf{RBL}$) which is the logic of residuated basic algebras\footnotetext[1]{Through our communication with Hiroakira Ono (Japan advanced Institute of Science and Technology), we know that Majid Alizadeh (University of Tehran) has the following unpublished result: basic propositional logic with residuation is a conservative extension of basic propositional logic. Ono's idea is to formalize basic logic with residuation as a sequent system which is an extension of $\mathsf{GL}$ in \cite{GO10}. However, the results in our paper are obtained independently. }. We present the algebraic system $\mathsf{S_{RBL}}$ for residuated basic algebra in terms of which we show that $\mathsf{RBL}$ is a conservative extension of $\mathsf{BPL}$. We present a Gentzen-style sequent formalization $\mathsf{L_{RBL}}$ of $\mathsf{S_{RBL}}$ which makes it easy to compare $\mathsf{RBL}$ with other substructural logics.

Basic algebra is the algebra for basic propositional logic (\cite{Ish02}). We define residuated basic algebra as bounded distributive lattice order residuated groupoid (cf. \cite{Bus11}) enriched with weakening and restricted contraction $(\mathrm{c_r})$ $a\cdot b\leq (a\cdot b) \cdot b$. It turns out that the reduct of residuated basic algebra restricted to the basic algebra type is in fact basic algebra (cf. theorem \ref{thm:arba}). Note that our residuated basic algebra does not contain the unit element of its residuated groupoid reduct. For showing that the residuated basic logic is a conservative extension of $\mathsf{BPL}$, we introduce the algebraic system $\mathsf{S_{RBL}}$ for residuated basic algebra. The main method for proving the conservative extension is the relational semantics of non-associative Lambek calculus, and some techniques and notations are taken from \cite{KF13}.

The Gentzen-style sequent calculus $\mathsf{L_{RBL}}$ is obtained by enriching the sequent calculus of $\mathsf{DFNL}$ ($\mathsf{FNL}$ with distributive law \cite{BusF09}) with weakening and restricted contraction rules. The logic $\mathsf{FNL}$ (\cite{Bus10}) is obtained by adding all lattice operations and their corresponding rules to the logic $\mathsf{NL}$ (non-associative Lambek calculus originally introduced by Lambek \cite{Lam61}) which is strongly complete with respect to residuated groupoids (\cite{Bus10}). The associative variant $\mathsf{L}$ (Lambek calculus) of $\mathsf{NL}$ was also introduced by Lambek \cite{Lam58}.
The formalization style of our system $\mathsf{L_{RBL}}$ was first independently developed by Dunn \cite{Dunn73} and Mints \cite{Mints76} for the positive relevant logic $\mathsf{R}^+$, and later used by Kozak \cite{Koz09} to prove the finite model property for $\mathsf{DFL}$ (distributive full Lambek calculus). Moreover, we prove that the cut elimination, subformula property and disjunction property hold for $\mathsf{L_{RBL}}$. Finally, we discuss the relationships between $\mathsf{BPL}$, $\mathsf{RBL}$ and other logics. As a consequence, the fragment of $\mathsf{RBL}$ restricted to the language of $\mathsf{BPL}$ is equivalent to the implicational fragment of $\mathsf{BPL}$ in \cite{Kik01}.

Now let us give some preliminaries on basic propositional logic, which can be found in \cite{AA04}, \cite{AR98}, \cite{SO97} and \cite{visser81}. The language $\mc{L}_{\mathrm{BPL}}$ of basic propositional logic consists of a set $\mathsf{Prop}$ of propositional letters and connectives $\wedge,\vee,\imp,\bot$ and $\top$. The set of $\mc{L}_{\mathrm{BPL}}$-formulae is defined recursively by the following rule:
\begin{center}
$A::= p\mid \bot\mid\top\mid A\vee A\mid A\wedge A\mid A\imp A$
\end{center}
where $p\in \mathsf{Prop}$. Define $\neg A:=A\imp\bot$, and $A\bimp B:= (A\imp B)\wedge(B\imp A)$.

A BPL-frame is a pair $\mf{F}=(W,R)$ where $W$ is a nonempty sets of states, and $R\sub W^2$ is a transitive relation, i.e., $\forall x,y,z\in W(xRy\wedge yRz\imp xRz)$.
A BPL-model is a tuple $\mf{M}=(W,R,V)$ where $(W,R)$ is a BPL-frame and $V:\mathsf{Prop}\imp \wp(S)$ is a valuation satisfying the following `persistency' condition: for each propositional letter $p$,
if $w\in V(p)$ and $wRu$, then $u\in V(p)$.
The definition of satisfaction relation $\mf{M},w\models A$ is defined as usual (\cite{visser81} and \cite{CZ97}). Especially, for implication, we have the following clause:
\begin{center}
$\mf{M},w\models A\imp B$ iff for all $v\in W$ with $wRv$, $\mf{M},v\models A$ implies $\mf{M},v\models B$.
\end{center}
The notion of frame validity $\mf{F}\models A$ is also defined as usual. Note that the truth of every $\mc{L}_{\mathrm{BPL}}$-formula is persistent: if $\mf{M},w\models A$ and $wRv$, then $\mf{M},v\models A$.

Visser introduced the natural deduction for $\mathsf{BPL}$ in \cite{visser81}. It can be extended to intuitionistic logic ($\mathsf{Int}$) and formal provability logic ($\mathsf{FPL}$). It is already known that, via G\"{o}del's translation, $\mathsf{BPL}$ is embedded into the modal logic $\mathsf{K4}$, $\mathsf{FPL}$ into the G\"{o}del-L\"{o}b modal logic $\mathsf{GL}$, and $\mathsf{Int}$ into the modal logic $\mathsf{S4}$. The Hilber-style axiomatization of $\mathsf{BPL}$, given by Ono and Suzuki in \cite{SO97}, consists of the following axioms and rules:
\begin{itemize}
\item[]($\mathrm{1}$)~$A\imp A$
\item[]($\mathrm{2}$) ~$A\imp(B\imp A)$
\item[]($\mathrm{3}$) ~$(A\imp B)\wedge(B\imp C)\imp(A\imp C)$
\item[]($\mathrm{4}$) ~$(A\imp C)\wedge(B\imp C)\imp (A\vee B \imp C)$
\item[]($\mathrm{5}$) ~$A\wedge B\imp A$
\item[]($\mathrm{6}$) ~$A\wedge B\imp B$
\item[]($\mathrm{7}$) ~$A\imp A\vee B$ 
\item[]($\mathrm{8}$) ~$B\imp A\vee B$
\item[]($\mathrm{9}$) ~$A\imp (B\imp A\wedge B)$
\item[]($\mathrm{10}$) ~$(A\imp B)\wedge (A\imp C)\imp (A\imp B\wedge C)$
\item[]($\mathrm{11}$) ~$A\wedge(B\vee C)\imp (A\wedge B)\vee (A\wedge C)$
\item[]($\mathrm{12}$) ~$\bot\imp A$
\item[]($\mathrm{MP}$)~ from $A$ and $A\imp B$ infer $B$.
\end{itemize}
By $\vdash_\mathsf{BPL} A$ we mean that $A$ is provable (or a theorem) in the above Hilbert-style system.
The following completeness theorem holds for $\mathsf{BPL}$:
\begin{theorem}[\cite{SO97}, \cite {visser81}]
For all $\mc{L}_{\mathrm{BPL}}$-formulas $A$, $\vdash_\mathsf{BPL} A$ iff $\mf{F}\models A$ for all BPL-frames $\mf{F}$ iff $\mf{M}\models A$ for all BPL-models $\mf{M}$.
\end{theorem}

Just as Heyting algebras for intuitionistic logic, there is also a variety of algebras for $\mathsf{BPL}$, which is called the variety of basic algebras.

\begin{definition}[\cite{AA04}]
A basic algebra $\mathbf{A}=(\mathsf{A},\wedge,\vee,\top,\bot,\rightarrow)$ is an algebra such that $(\mathsf{A},\wedge,\vee,\top, \bot)$ is a bounded distributive lattice and $\rightarrow$ is a binary operation over $A$ satisfying the following conditions: for all $a,b,c\in \mathsf{A}$,
\begin{itemize}
\item[(1)] $a\rightarrow(b\wedge c)=(a\rightarrow b)\wedge(a\rightarrow c)$.
\item[(2)] $(b\vee c)\rightarrow a=(b\rightarrow a)\wedge(c\rightarrow a)$.
\item[(3)] $a\rightarrow a= \top$.
\item[(4)] $a\leq \top \rightarrow a$.
\item[(5)] $(a\rightarrow b)\wedge(b\rightarrow c)\leq a \rightarrow c$.
\end{itemize}
The operation $\imp$ in a basic algebra  is said to be {\em basic implication}.
\end{definition}

\begin{fact}[\cite{AA04}]
For any basic algebra $\mathbf{A}=(\mathsf{A},\wedge,\vee,\top,\bot,\rightarrow)$ and $a,b,c\in \mathsf{A}$,
\begin{itemize}
\item[$\mathrm{(1)}$] if $a\leq b$, then $c\imp a\leq c\imp b$, $b\imp c\leq a\imp c$ and $a\imp b=\top$.
\item[$\mathrm{(2)}$] if $a\wedge b\leq c$, then $a\leq b\imp c$.
\end{itemize}
\end{fact}

Moreover, the relationship between basic algebra and Heyting algebra is clear by the following result from \cite{AA04}: a basic algebra $\mathbf{A}$ is a Heyting algebra iff $a=\top\imp a $ for all $a\in \mathsf{A}$. Hence every Heyting algebra is a basic algebra.

Given a basic algebra $\mathbf{A}$ and an $\mc{L}_{\mathrm{BPL}}$-formula $A$, we say that $A$ is {\em valid} in $\mathbf{A}$ (notation: $\mathbf{A}\models A$), if the equation $A=\top$ holds in $\mathbf{A}$, i.e., the formula $A$ denotes the top element in the algebra under any assignment of propositional letters in $A$. Let $\mathsf{\mathbb{BLA}}$ be the class of all basic algebras. We mean by $\mathsf{\mathbb{BLA}}\models A$ that $A$ is valid in every basic algebra. It's already known that $\mathsf{BPL}$ is complete with respect to $\mathsf{\mathbb{BLA}}$ (\cite{AA04}), i.e., for all $\mc{L}_{\mathrm{BPL}}$-formulae $A$, $\vdash_{\mathsf{BPL}} A$ iff $\mathsf{\mathbb{BLA}}\models A$.

\section{Basic Implication and Residuation}
It is well-known that the Heyting implication $\imp_H$ is the residual of $\wedge$, i.e., for any $a,b,c$ in a Heyting algebra,
$c\leq a\imp_H b$ iff $a\wedge c\leq b$.
However, the basic implication is not the residual of $\wedge$ in basic algebra. Let us introduce an binary operator $\cdot$ the right residual of which is supposed to be the basic implication. Then an algebra $(\mathsf{A},\cdot,\imp,\leftarrow,\leq)$ is said to be a {\em residuated groupoid}, if $(\mathsf{A},\leq)$ is a poset, and $\cdot,\imp$ and $\leftarrow$ are binary operators satisfying the following residuation law:
\begin{center}
(RES) $a\cdot b\leq c$ iff $b\leq a\imp c$ iff $a\leq c\leftarrow b$.
\end{center}
The associativity of the operator $\cdot$ is not assumed in residuated groupoid.

\begin{definition}
A {\em residuated basic algebra} ($\mathbf{RBA}$) is an algebra $\mathbf{A}=(\mathsf{A},\wedge,\vee,\top,\bot,\rightarrow,\leftarrow,\cdot)$ such that $(\mathsf{A},\wedge,\vee,\top,\bot)$ is a bounded distributive lattice and $(\mathsf{A},\rightarrow, \leftarrow, \cdot,\leq)$ is a residuated groupoid satisfying the following axioms: for all $a,b,c\in \mathsf{A}$,
\begin{align*}
(\mathrm{w}_1)~& a\cdot\top \leq a;~~
(\mathrm{w}_2)~\top\cdot a\leq a;~~
(\mathrm{c_r})~a\cdot b\leq (a\cdot b)\cdot b
\end{align*}
where $\leq$ is the lattice order. Let $\mb{RBA}$ be the class of all residuated basic algebras.
\end{definition}
For any residuated basic algebra, it is easy to show (i) $a\imp b = \bigvee\{x\in \mathsf{A}:a\cdot x\leq b\}$; and (ii) $a\leftarrow b = \bigvee\{x\in \mathsf{A}:x\cdot b\leq b\}$ (the least upper bound).

\begin{remark}\label{remark:rba}
The element $\top$ in a residuated basic algebra $\mathbf{A}$ is not the unit of $\mathbf{A}$'s residuated groupoid reduct. If we enrich residuated basic algebra by constant $1$ (the unit of residuated groupoid), then by $(\mathrm{w}_1)$ and $(\mathrm{w}_2)$, one can easily prove that $1=\top$. It follows that $\top\imp a\leq a$, ans so the implication $\imp$ becomes Heyting.
\end{remark}

\begin{proposition}\label{prop:rbamon}
For any residuated basic algebra $\mathbf{A}$ and $a,b,c\in \mathsf{A}$,
\begin{itemize}
\item[(i)] $(b\vee c)\cdot a \leq b\cdot a \vee  c\cdot a$.
\item[(ii)] $a\cdot (b\cdot c)\leq (a\cdot b)\cdot c$.
\item[(iii)] if $a\leq b$, then $c\cdot a\leq c\cdot b$ and $a\cdot c\leq b\cdot c$.
\item[(iv)] if $a\leq b$, then $c\imp a\leq c\imp b$ and $b\imp c\leq a\imp c$.
\item[(v)] if $a\leq b$, then $a\leftarrow c\leq b\leftarrow a$ and $c\leftarrow b\leq c\leftarrow a$.
\end{itemize}
\end{proposition}
\begin{proof}
It is easy to see that (iii)-(iv) are monotonicity laws which hold in all residuated groupoid. From (iii), we can easily derive that $a\leq b$ and $c\leq d$ imply $a\cdot c \leq b\cdot d$. For (i), since $b\cdot a \leq (b\cdot a)\vee (c\cdot a )$ and $c\cdot a\leq (b\cdot a)\vee (c\cdot a)$, we get $b\leq  ((b\cdot a)\vee (c\cdot a))\leftarrow a$ and $c\leq  ((b\cdot a)\vee (c\cdot a))\leftarrow a$. Hence $(b\vee c)\leq ( (b\cdot a)\vee (c\cdot a))\leftarrow a$, which yields $(b\vee c)\cdot a \leq (b\cdot a)\vee (c\cdot a)$. For (ii), first by ($\mathrm{w_1}$) and (iii), we have $b\cdot c\leq b$ and $b\cdot c\leq c$. By (iii) again, we get $a\cdot (b\cdot c)\leq a\cdot b$. Since $b\cdot c\leq c$, we get $(a\cdot (b\cdot c))\cdot (b\cdot c)\leq (a\cdot b)\cdot c$. By $(\mathrm{c_r})$, we get $a\cdot (b\cdot c)\leq (a\cdot b)\cdot c$.
\end{proof}

Moreove, for any (residuated) basic algebra $\mathbf{A}$ and $a,b\in \mathsf{A}$, it is easy to check that $a=b$ iff $\forall x\in \mathsf{A}(x\leq a\leftrightarrow x\leq b)$. This gives a way for showing that an equation $a=b$ hold in $\mathbf{A}$.
\begin{theorem}\label{thm:arba}
Let $\mathbf{A} = (\mathsf{A},\wedge,\vee,\top,\bot,\rightarrow,\leftarrow, \cdot)$ be a residuated basic algebra. Then $(\mathsf{A},\wedge,\vee,$ $\top,\bot,\rightarrow)$ is a basic algebra.
\end{theorem}
\begin{proof}
It suffices to show that all axioms of basic algebra hold in every residuated basic algebra. Obviously, $(\mathsf{A},\wedge,\vee,\top,\bot)$ is a bounded distributive lattice.

(i) Assume $x\leq a\rightarrow (b\wedge c)$. By residuation, $a\cdot x\leq b\wedge c$. Since $b\wedge c\leq b$ and $b\wedge c\leq c$, we get $a\cdot x\leq b$ and $a \cdot x \leq c$. Thus $x\leq a\rightarrow b$ and $x\leq a\rightarrow c$. Hence $x\leq (a\rightarrow b)\wedge (a\rightarrow c)$. Conversely, assume $x\leq (a\rightarrow b)\wedge (a\rightarrow c)$. Then $x\leq a\rightarrow b$ and $x\leq a\rightarrow c$. By residuation, $ a\cdot x\leq b$ and $a\cdot x\leq c$. Thus $a\cdot x\leq b\wedge c$, and so $x\leq (a\rightarrow b\wedge c)$. Hence $a\rightarrow (b\wedge c) =(a\rightarrow b)\wedge (a\rightarrow c)$.

(ii) Assume $x\leq (b\vee c) \rightarrow a $. By residuation, we get $(b\vee c )\cdot x \leq a $. Since $b\leq b\vee c$, we have $b\cdot x\leq (b\vee c)\cdot x$. Then $b\cdot x \leq a $. Similarly we get $c\cdot x \leq a$. Thus by residuation, we have $x\leq b\rightarrow a $ and $x\leq c\rightarrow a $. Hence $x\leq (b\rightarrow a)\wedge (c\rightarrow a)$. Conversely, assume $x\leq (b\rightarrow a)\wedge (c\rightarrow a)$. Then $x\leq b\imp a$ and $x\leq c\imp a$.
By residuation, $b\cdot x\leq a$ and $c\cdot x \leq a$. Then $b\cdot x\vee c\cdot x\leq a$. By proposition \ref{prop:rbamon} (i), we get $(b\vee c)\cdot x\leq b\cdot x \vee c \cdot x $, and so we get $(b\vee c)\cdot x \leq a$. Hence $(b\vee c)\rightarrow a =(b\rightarrow a)\wedge (c\rightarrow a)$.

(iii) Obviously $a\rightarrow a \leq \top$. Since $a \cdot \top \leq a$, we get $\top \leq a\rightarrow a$. Hence $a\imp a =\top$. Again, since $\top \cdot a \leq a$, by residuation we get $a \leq \top \rightarrow  a$.

(iv) Assume $x\leq (a\rightarrow b) \wedge (b\rightarrow c)$. Then $x\leq(a\rightarrow b)$ and $x\leq (b \rightarrow c)$, whence $a\cdot x \leq b$ and $b\cdot x\leq c$. Then by \ref{prop:rbamon} (iii), we get $(a\cdot x)\cdot x \leq b\cdot x$. Hence $(a\cdot x)\cdot x \leq c$. Since $a\cdot x \leq (a\cdot x)\cdot x$, we have $a\cdot x \leq c$, which yields $ x\leq a\rightarrow c$.  Hence $(a\rightarrow b)\wedge (b\rightarrow c)\leq (a\rightarrow c)$.
\end{proof}

Now let us define the residuated basic logic ($\mathsf{RBL}$) of residuated basic algebras. The language $\mc{L}_{\mathrm{RBL}}$ is the extension of $\mc{L}_{\mathrm{BPL}}$ by adding binary operators $\cdot$ and $\leftarrow$. The set of all $\mc{L}_{\mathrm{RBL}}$-formulae is defined recursively by:
\begin{center}
$A::=p\mid\bot\mid\top\mid A\wedge A\mid A\vee A\mid A\cdot A\mid A\imp A\mid A\leftarrow A$
\end{center}
where $p\in\mathsf{Prop}$. An {\em assignment} $\sigma$ in a residuated basic algebra $\mathbf{A}$ is a homomorphism from the $\mc{L}_{\mathrm{RBL}}$-formula algebra to $\mathbf{A}$. An $\mc{L}_{\mathrm{RBL}}$-formula $A$ is {\em valid} in $\mathbf{A}$ (notation: $\mathbf{A}\models A$), if
$\sigma(A)=\top$ for all assignments $\sigma$ in $\mathbf{A}$. By $\mb{RBA}\models A$ we mean that $\mathbf{A}\models A$
for all residuated basic algebras. Thus we define the residuated basic logic $\mathsf{RBL} = \{A\mid\mb{RBA}\models A$ and $A$ is an $\mc{L}_{\mathrm{RBL}}$-formula$\}$.

\section{Conservative Extension}
A logic $L_2$ is said to be a {\em conservative extension} of $L_1$, if every formula provable in $L_1$ is provable in $L_2$.
In this section, we introduce an algebraic system $\mathsf{S_{RBL}}$ for residuated basic algebra, and show that  $\mathsf{S_{RBL}}$ is a conservative extension of $\mathsf{BPL}$ (cf. theorem \ref{thm:con}). It follows that $\mathsf{RBL}$ is a conservative extension of $\mathsf{BPL}$.

Simple $\mathsf{S_{RBL}}$-sequents are expressions of the form $A \Rightarrow B$ where $A$ and $B$ are $\mc{L}_{\mathrm{RBL}}$-formulae.
The algebraic system $\mathsf{S_{RBL}}$ for residuated basic algebras consists of the following axioms and rules:
\begin{displaymath}
(\mathrm{Id})~A \Rightarrow A\quad (\bot)~\bot\Rightarrow A\quad (\top)~A \Rightarrow \top \quad(\mathrm{Cut})~ \frac{A\Rightarrow B \quad B\Rightarrow C}{A\Rightarrow C}
\end{displaymath}
\begin{displaymath}
(\mathrm{D})~ A\wedge(B\vee C)\Rightarrow (A\wedge B)\vee (A\wedge C)
\end{displaymath}
\begin{displaymath}
(\mathrm{W}_l)~ A\cdot \top \Rightarrow A\quad (\mathrm{W}_r)~ \top\cdot A\Rightarrow A
\quad
\mathrm{(RC)}~ A\cdot B\Rightarrow (A\cdot B)\cdot B
\end{displaymath}
\begin{displaymath}
(\mathrm{R1})~ \frac{A\cdot B \Rightarrow C}{B\Rightarrow A\rightarrow C}\quad
(\mathrm{R2})~\frac{B \Rightarrow A\rightarrow C}{A\cdot B\Rightarrow C}
\end{displaymath}
\begin{displaymath}
(\mathrm{R3})~ \frac{A\cdot B \Rightarrow C}{A\Rightarrow C\leftarrow B}\quad
(\mathrm{R4}) ~\frac{A \Rightarrow C\leftarrow B}{A\cdot B\Rightarrow C}
\end{displaymath}
\begin{displaymath}
\mathrm{(\wedge L)}~\frac{A_i\Rightarrow B}{A_1\wedge A_2 \Rightarrow B},~{i\in\{1,2\}}\quad
\mathrm{(\wedge R)}~\frac{C\Rightarrow A\quad C \Rightarrow B}{C \Rightarrow A\wedge B}
\end{displaymath}
  \begin{displaymath}
(\mathrm{\vee L})~\frac{A\Rightarrow C\quad B\Rightarrow C}{A\vee B \Rightarrow C}\quad
(\mathrm{\vee R)}~ \frac{C \Rightarrow A_i}{C \Rightarrow A_1\vee A_2},~{i\in\{1,2\}}
   \end{displaymath}
By $\vdash_{\mathsf{S_{RBL}}}A\Imp B$ we mean that $A\Imp B$ is provable in $\mathsf{S_{RBL}}$.
We say a $\mathsf{S_{RBL}}$-formula $A$ is {\em provable} in $\mathsf{S_{RBL}}$, if $\vdash_{\mathsf{S_{RBL}}} \top\Leftrightarrow A$.

It is easy to prove that the system $\mathsf{S_{RBL}}$ is complete with respect to the class of all bounded distributive lattice order residuated groupoid enriched with weakening and restriced contraction ($\mathrm{c_r}$), which is exactly the class $\mathbb{RBA}$ of all residuated basic algebras. Let us explain some basic notions. An assignment $\sigma$ in a residuated basic algebra $\mathbf{A}$ is a homomorphism from the $\mc{L}_{\mathrm{RBL}}$--formula algebra into $\mathbf{A}$. A sequent $A\Imp B$ is true under an assignment $\sigma$ in a residuated basic algebra $\mathbf{A}$, if $\sigma(A)\leq \sigma(B)$ in $\mathbf{A}$.
We say that $A\Imp B$ is valid in $\mathbf{A}$, if $A\Imp B$ is true under all assignments in $\mathbf{A}$. By $\mb{RBA}\models A\Imp B$ we mean that $A\Imp B$ is valid in all residuated basic algebras.
The completeness means that $\vdash_{\mathsf{S_{RBL}}} A\Imp B$ iff $\mb{RBA}\models A\Imp B$.

Let $\mathsf{S^*_{RBL}}$ be the system obtained from $\mathsf{S_{RBL}}$ by replacing $(\mathrm{W}_l)$, $(\mathrm{W}_r)$ and $(\mathrm{RC})$ by the following three axioms respectively:
\begin{displaymath}
(\top^1)~ \top\Rightarrow A\rightarrow A~~(\top^2)~ A\Rightarrow \top \rightarrow A
~~
(\mathrm{Tr})~ (A\imp B)\wedge(B\imp C)\Imp A\imp C.
\end{displaymath}
\begin{theorem}\label{thm:equi}
The system $\mathsf{S^*_{RBL}}$ is equivalent to $\mathsf{S_{RBL}}$.
\end{theorem}
\begin{proof}
It suffices to show that $(\mathrm{W}_l)$, $(\mathrm{W}_r)$ and $(\mathrm{RC})$ are equivalent to $(\top^1)$, $(\top^2)$ and $(\mathrm{Tr})$ respectively. By the proof of theorem \ref{thm:arba}, we know that $(\mathrm{W}_l)$, $(\mathrm{W}_r)$ and $(\mathrm{RC})$ imply $(\top^1)$, $(\top^2)$ and $(\mathrm{Tr})$ respectively. Conversely, assume $\top\Rightarrow A\rightarrow A$. By ($\mathrm{R3}$), we have $A\cdot \top\Rightarrow A$. Assume $A\Rightarrow \top \rightarrow A$. By ($\mathrm{R3}$), we have $\top \cdot A\Rightarrow A$. Finally, assume $(\mathrm{Tr})$ holds. Since $A\cdot B\Rightarrow  A\cdot B$ and $(A\cdot B)\cdot B \Rightarrow (A\cdot B)\cdot B$, by residuation, we obtain $B\Rightarrow A\rightarrow (A\cdot B)$ and $B\Rightarrow (A\cdot B) \rightarrow ((A\cdot B)\cdot B)$. Hence $B\Rightarrow (A\rightarrow (A\cdot B))\wedge((A\cdot B)\rightarrow ((A\cdot B)\cdot B))$. Since $(A\imp (A\cdot B))\wedge((A\cdot B)\rightarrow ((A\cdot B)\cdot B))\Rightarrow A\rightarrow ((A\cdot B)\cdot B)$, one gets $B\Rightarrow A\rightarrow ((A\cdot B)\cdot B)$. By ($\mathrm{R2}$), we have $A\cdot B\Rightarrow (A\cdot B)\cdot B$.
\end{proof}

Now let us prove that $\mathsf{S^*_{RBL}}$ is a conservative extension of $\mathsf{BPL}$ (cf. theorem \ref{thm:cons}). In the whole proof, we use the relational model from \cite{KF13} ($\mathsf{S^*_{RBL}}$-model in definition \ref{def:arblmodel}), in which a tenary relation is used to interpret binary operators. The main tecnique used in the proof is to construct a $\mathsf{S^*_{RBL}}$-model $\mf{J}^\mf{M}$ from a BPL-model $\mf{M}$. By our construction, the model $\mf{J}^\mf{M}$ is a model for the system $\mathsf{S^*_{RBL}}$ (cf. lemma \ref{lemma:arbasound}). This fact is enough for showing that every $\mc{L}_{\mathrm{BPL}}$-formula $A$ provable in $\mathsf{S^*_{RBL}}$ (i.e., $\vdash_{\mathsf{S^*_{RBL}}} \top\Leftrightarrow A$) is also provable in $\mathsf{BPL}$. Let us introduce the relational semantics first.

\begin{definition}\label{def:arblmodel}
An {\em residuated basic frame} is a pair $\mc{F}=(W,R)$ where $W$ is a nonempty set of states, and $R\sub W^3$ is a tenary relation.
An $\mathsf{S^*_{RBL}}$-model is a triple $\mf{J}=( W,R,V)$, where $(W,R)$ is a residuated basic frame and $V:\mathsf{Prop}\imp\wp(W)$ (the powerset of $W$) is a valuation function.
\end{definition}

\begin{definition}
The satisfiability relation $\mf{J},a\models A$ between an $\mathsf{S^*_{RBL}}$-model $\mf{J}$ with state $a$ and a $\mc{L}_{\mathrm{RBL}}$-formula $A$ is defined recursively as follows:
\begin{itemize}
\item $\mf{J},a\models p$, if $a\in V(p)$.
\item $\mf{J},a\models \top$ and $\mf{J},a\not\models \bot$.
\item $\mf{J},a\models A\cdot B$, if there exist $a_2, a_3\in W$ such that $R(a,a_2,a_3)$, $\mf{J},a_2\models A$ and $\mf{J},a_3\models B$.
\item $\mf{J},a \models A\leftarrow B$, if for all $a_1, a_3\in W$ with $R(a_1,a,a_3)$, $\mf{J},a_3\models B$ implies $\mf{J},a_1\models A$.
\item $\mf{J},a\models A\rightarrow B$, if for all $a_1, a_2\in W$ with $R(a_1,a_2,a)$,  $\mf{J},a_2 \models A$ implies $\mf{J},a_1\models B$.
\item $\mf{J},a\models A\wedge B$, if $\mf{J},a\models A$ and $\mf{J},a\models B$.
\item $\mf{J},a\models A\vee B$, if $\mf{J},a\models A$ or $\mf{J},a\models B$.
\end{itemize}
An $\mc{L}_{\mathrm{RBL}}$-formula $A$ is {\em satisfiable}, if there exist a relational model $\mf{J}$ and a state $a$ in $\mf{J}$ such that $\mf{J}, a\models A$. We say that $A$ is {\em true} in $\mf{J}$ (notation: $\mf{J}\models A$), if $\mf{J}, a\models A$ for all states $a$ in $\mf{J}$. An $\mathsf{S^*_{RBL}}$-sequent $A\Imp B$ is {\em true} at $a$ in an $\mathsf{S^*_{RBL}}$-model $\mf{J}$ (notation: $\mf{J},a\models A\Imp B$), if $\mf{J},a\models A$ implies $\mf{J},a\models B$. A sequent $A\Imp B$ is true in $\mf{J}$ (notation: $\mf{J}\models A\Imp B$), if $\mf{J},a\models A\Imp B$ for all states $a$ in $\mf{J}$.
\end{definition}

Now we construct a tenary $\mathsf{S^*_{RBL}}$-model from a binary BPL-model, which will be shown to be a model for the system $\mathsf{S^*_{RBL}}$.

\begin{definition}
Given a BPL-model $\mf{M}=(W, R, V)$, define the $\mathsf{S^*_{RBL}}$-model $\mf{J}^\mf{M}=(W', R', V')$ constructed from $\mf{M}$ as follows:
\begin{itemize}
\item[(1)] $W'=\{a_1,a_2| a\in W\}$
\item[(2)] $R'=\{\langle b_1,b_2,a_1\rangle, \langle b_2,b_1, a_1\rangle, \langle b_1, b_2, a_2\rangle, \langle b_2,b_1,a_2\rangle~|~aRb\}$
\item[(3)] $V'(p)=\{a_i\in W:i\in\{1,2\}~\&~a\in V(p)\}$, for each $p\in\mathsf{Prop}$.
\end{itemize}
\end{definition}
\begin{center}
\includegraphics{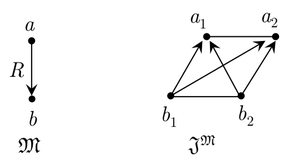}
\end{center}
Henceforth we show some properties of the induced tenray models.
\begin{lemma}\label{lemma:truth1}
Let $\mf{M}=(W,R,V)$ be a BPL-model and $\mf{J}^\mf{M}=(W',R',V')$. For any $\mc{L}_{\mathrm{BPL}}$-formula $A$, $a\in W$ and $a_i\in W'$, $\mf{M},a \models A$ iff $\mathfrak{J}^\mf{M},a_i\models A$.
\end{lemma}
\begin{proof}
By induction on the complexity of $A$. The cases of propositional letters, disjunction and conjunction are easy. We consider only the case of implication. Let $A=A_1\rightarrow A_2$. Assume that $\mathfrak{J}^\mf{M}, a_i\not\models A_1 \rightarrow A_2$. Suppose $a_i= a_1$ without loss of generality. Then there exist $a, b\in W$ such that $Rab$ and $R'(b_1,b_2,a_1)$. By the construction of $\mathfrak{J}^\mf{M}$, we have $\mathfrak{J}^\mf{M},b_2\models A_1$ and  $\mathfrak{J}^\mf{M},b_1\not\models A_2$. By induction hypothesis, $\mathfrak{M},b\models A_1$ and $\mathfrak{M}, b\not\models A_2$. Then $\mathfrak{M},a \not\models A_1\rightarrow A_2$.
Conversely, assume $\mathfrak{M},a\not\models A_1 \rightarrow A_2$. Then there exists $b\in W$ such that $Rab$, $\mathfrak{M},b\models A_1$ and $\mathfrak{M}, b\not\models A_2$. By induction hypothesis, $\mathfrak{J}^\mf{M},b_2\models A_1$ and $\mathfrak{J}^\mf{M}, b_1\not\models A_2$. Since $Rab$, we have $R'(b_1,b_2,a_1)$. Hence $\mathfrak{J}^\mf{M}, a\not\models A_1\rightarrow A_2$.
\end{proof}

\begin{corollary}\label{coro:bpc}
For any $\mc{L}_{\mathrm{BPL}}$-formula $A$ and model $\mf{M}$, $\mf{M}\models A$ iff $\mf{J}^\mf{M}\models A$.
\end{corollary}

\begin{lemma}\label{lemma:arbasound}
For any BPL-model $\mf{M}$ and sequent $A\Imp B$ where $A,B$ are $\mc{L}_{\mathrm{BPL}}$-formulae,
if $\vdash_{\mathsf{S^*_{RBL}}} A\Imp B$, then $\mathfrak{J}^\mf{M}\models A \Rightarrow B$.
\end{lemma}
\begin{proof}
We need to show that all axioms and rules are admissible in $\mathfrak{J}^\mf{M}$. Axioms except $(\top^1),(\top^2),(\mathrm{Tr})$ are easily shown to be true in $\mf{J}^\mf{M}$. We prove that the three axioms are true in $\mf{J}^\mf{M}$ as follows.

$(\top^1)$ It is obviuos that $\mf{M},a\models A\imp A$ for any $a$ in $\mf{M}$. By lemma \ref{lemma:truth1}, we get $\mf{J},a_i\models A\imp A$. Hence $\mf{J},a_i\models \top\Imp A\imp A$.

$(\top^2)$ Assume $\mf{J},a_i\models A$. By lemma \ref{lemma:truth1}, we get $\mf{M},a\models A$. For any state $b$ in $\mf{M}$ such that $Rab$, we have $\mf{M},b\models A$. Then $\mf{M},a\models \top\imp A$. By lemma \ref{lemma:truth1}, we get $\mf{J}^\mf{M},a_i\models \top\imp A$. Hence $\mf{J}^\mf{M},a_i\models A\Imp \top\imp A$.

$(\mathrm{Tr})$ Assume $\mf{J}^\mf{M}, a_i\models (A\imp B)\wedge(B\imp C)$ but $\mf{J}^\mf{M},a_i\not\models A\imp C$. Then there exist $b_1,b_2\in W'$ such that $R(b_1,b_2,a_i)$, $\mf{J}^\mf{M},b_2\models A$ and $\mf{J}^\mf{M},b_1\not\models C$. By $\mf{J}^\mf{M},b_2\models A$, $\mf{J}^\mf{M},a_i\models A\imp B$ and $R(b_1,b_2,a_i)$, we get $\mf{J}^\mf{M},b_1\models B$. Since $\mf{J}^\mf{M},b\models B\rightarrow C$, by lemma \ref{lemma:truth1}, we get $\mf{M},a\models B\imp C$, $\mf{M},b\models B$, $\mf{M},b\not\models C$. Since $R(b_1,b_2,a_i)$, we get $Rab$. Then by $\mf{M},a\models B\imp C$ and $\mf{M},b\models B$, we get $\mf{M},b\models C$, a contradiction.

Consider the rules $(\wedge \mathrm{L}), (\wedge \mathrm{R}), (\vee \mathrm{L}), (\vee \mathrm{R})$ and $(\mathrm{Cut})$. We check only $(\wedge \mathrm{L})$ and $(\mathrm{Cut})$.
Other rules are shown to be admissible in $\mf{J}^\mf{M}$ similarly. First, for $(\wedge \mathrm{L})$, assume $\mf{J}^\mf{M},a_1\models A_i\Imp B$ and $\mf{J}^\mf{M},a_1\models A_1\wedge A_2$. Then $a_1\models A_i$, and so $a_1\models B$. Hence $\mf{J}^\mf{M},a_1\models A_1\wedge A_2\Imp B$. Second, for $(\mathrm{Cut})$, assume that $a_1\models A\Imp B$ and $a_1\models B\Imp C$ and $a_1\models C$. By definition, we get $a_1\models C$. Hence $a_1\models A\Imp C$.

Now it suffices to show the admisibility of residuation rules.

$(\mathrm{R1})$ Assume $\mathfrak{J}^\mf{M}\models A\cdot B\Rightarrow C$. Suppose
$\mf{J}^\mf{M},a_1\models B$ but $\mf{J}^\mf{M},a_1\not\models A\rightarrow C$ for some $a_1\in W'$. Then there are $b_1,b_2$ such that $R'(b_1,b_2,a_1)$, $\mf{J}^\mf{M},b_2\models A$ and $\mf{J}^\mf{M},b_1\not\models C$. Thus $\mf{J}^\mf{M},b_1\models A\cdot B$. By assumption, $\mf{J}^\mf{M},b_1\models C$, a contradiction.

$(\mathrm{R2})$ Assume $\mathfrak{J}^\mf{M}\models B\Rightarrow A\rightarrow C$. Suppose $\mf{J}^\mf{M},a_1\models A\cdot B$ but $\mf{J}^\mf{M},a_1\not\models C$. Then there exists $b_1\in W'$ such that $R'(a_1,a_2,b_1)$, $\mf{J}^\mf{M},a_2\models A$ and $\mf{J}^\mf{M},b_1\models B$. By assumption, $\mathfrak{J}^\mf{M},b_1\models A\rightarrow C$. By $\mf{J}^\mf{M},a_2\models A$, we get $\mf{J}^\mf{M},a_1\models C$, a contradiction.

$(\mathrm{R3})$ Assume $\mathfrak{J}^\mf{M}\models A\cdot B\Rightarrow C$. Suppose $\mf{J}^\mf{M},b_2\models A$ but $\mf{J}^\mf{M},b_2\not\models C\leftarrow B$. Then there exists $a_1\in W'$ such that $R'(b_1,b_2,a_1)$, $\mf{J}^\mf{M},a_1\models B$ and $\mf{J}^\mf{M},b_1\not\models C$.
Hence $\mf{J}^\mf{M},b_1\models A\cdot B$. By assumption, we have $\mf{J}^\mf{M},b_1\models C$, a contradiction.

$(\mathrm{R4})$ Assume $\mf{J}^\mf{M}\models A\Imp C\leftarrow B$. Suppose $\mf{J}^\mf{M},b_1\models A\cdot B$ but $\mf{J}^\mf{M},b_1\not\models C$.
Then there exists $a_1\in W'$ such that $R'(b_1,b_2,a_1)$, $\mf{J}^\mf{M},b_2\models A$ and $\mf{J}^\mf{M},a_1\models B$. By assumption, $\mf{J}^\mf{M},b_2\models C\leftarrow B$. Hence $\mf{J}^\mf{M},b_1\models C$, a contradiction.
\end{proof}

Now we can show the following conservative extension theorem.

\begin{theorem}\label{thm:cons}
For any $\mc{L}_{\mathrm{BPL}}$-formula $A$, $\vdash_{\mathsf{BPL}} A$ iff $\vdash_{\mathsf{S^*_{RBL}}}\top\Leftrightarrow A$.
\end{theorem}
\begin{proof}
Assume $\vdash_{\mathsf{BPL}} A$. Then $\mathsf{\mathbb{BLA}}\models A$ since $\mathsf{BPL}$ is complete with respect to the class $\mathsf{\mathbb{BLA}}$ of all basic algebras. By theorem \ref{thm:arba}, $\mathbb{RBA}\models A$.
Hence $\mathbb{RBA}\models \top\Leftrightarrow A$.
 Then by algebraic completeness of $\mathsf{S^*_{RBL}}$ with respect to $\mathbb{RBA}$, we have $\vdash_{\mathsf{S^*_{RBL}}} \top \Leftrightarrow A$.
Conversely, assume $\not\vdash_{\mathsf{BPL}} A$. By completeness of $\mathsf{BPL}$ with respect to the class of all BPL-models, there exists a  BPL-model $\mf{M}$ such that $\mf{M}\not\models A$. Then by corollary \ref{coro:bpc}, $\mf{J}^\mf{M}\not\models A$. Hence by the definition of truth of a sequent in a $\mathsf{S^*_{RBL}}$-model, $\mf{J}^\mf{M}\not\models \top\Imp A$. By lemma \ref{lemma:arbasound}, we obtain $\not\vdash_{\mathsf{S_{RBL}}}\top\Imp A$.
\end{proof}
By theorem \ref{thm:equi}, we obtain the following theorems immediately.

\begin{theorem}\label{thm:con}
For any $\mc{L}_{\mathrm{BPL}}$-formula $A$, $\vdash_{\mathsf{BPL}} A$ iff $\vdash_{\mathsf{S_{RBL}}} \top\Leftrightarrow A$.
\end{theorem}

\begin{corollary}
For every $\mc{L}_{\mathrm{BPL}}$-formula $A$, $\mathsf{\mathbb{BLA}}\models A$ iff $\mathbb{RBA}\models A$.
\end{corollary}

\begin{theorem}\label{thm:rblcon}
The logic $\mathsf{RBL}$ is a conservative extension of $\mathsf{BPL}$.
\end{theorem}

\section{A Sequent Calculus for $\mathsf{RBL}$}
The algebraic system $\mathsf{S_{RBL}}$ we introduced in section 3 for  residuated basic algebra is equivalent to $\mathsf{DFNL}$ with weakening and the restricted contraction rule $(\mathrm{RC})$. Let us introduce a sequent formalization $\mathsf{L_{RBL}}$ for $\mathsf{S_{RBL}}$, and show the cut elimination which implies the subformula property for $\mathsf{L_{RBL}}$. For general knowledge on sequent calculi for substructural logics, see Ono's survey paper \cite{Ono03}.

$\mc{L}_{\mathrm{RBL}}$-formula structures are defined recursively as follows: ($\romannumeral1$) every $\mc{L}_{\mathrm{RBL}}$-formula is a formula structure; ($\romannumeral2$) if $\Gamma$ and $\Delta$ are formula structures, then $\Gamma \odot \Delta$ and $\Gamma \owedge \Delta$ are formula structures. Each formula structure $\Gamma$ is associated with a formula $\mu(\Gamma)$ defined as follows:
(i) $\mu(A) = A$ for every $\mc{L}_{\mathrm{RBL}}$-formula $A$;
(ii) $\mu(\Gamma \odot\Delta) = \mu(\Gamma)\cdot \mu(\Delta)$;
(iii) $\mu(\Gamma \owedge \Delta) = \mu(\Gamma)\wedge \mu(\Delta)$.
Sequents are of the form $\Gamma \Rightarrow A$ where $\Gamma$ is an $\mc{L}_{\mathrm{RBL}}$-formula structure and $A$ is an $\mc{L}_{\mathrm{RBL}}$-formula.

The sequent calculus $\mathsf{L_{RBL}}$ consists of the following axioms and rules:
\begin{displaymath}
(\mathrm{Id}) ~ A \Rightarrow A \quad(\top)~ A\Rightarrow \top\quad (\bot)~ \bot\Rightarrow A
\end{displaymath}
\begin{displaymath}
(\mathrm{\imp L}) ~\frac{\Delta \Rightarrow A;  \quad \Gamma[B] \Rightarrow C}{\Gamma[\Delta \odot (A \imp B)] \Rightarrow C} \quad
(\mathrm{\imp R})~ \frac{A \odot \Gamma \Rightarrow B}{\Gamma \Rightarrow A \imp B}
\end{displaymath}
\begin{displaymath}
(\mathrm{\leftarrow L})~ \frac{\Gamma[A] \Rightarrow C;\quad \Delta \Rightarrow B}{\Gamma[(A\leftarrow B) \odot \Delta] \Rightarrow C} \quad
(\mathrm{\leftarrow R})~ \frac{\Gamma \odot B \Rightarrow A}{\Gamma \Rightarrow A\leftarrow B}
\end{displaymath}
\begin{displaymath}
(\mathrm{\cdot L})~\frac{\Gamma[A \odot B] \Rightarrow C}{\Gamma[A \cdot B] \Rightarrow C} \quad
(\mathrm{\cdot R}) ~\frac{\Gamma \Rightarrow A; \quad \Delta \Rightarrow B}{\Gamma \odot \Delta \Rightarrow A \cdot B}
\end{displaymath}
\begin{displaymath}
\mathrm{(\wedge L)}~\frac{\Gamma[A\owedge B]\Rightarrow C}{\Gamma[A\wedge B]\Rightarrow C}\quad
\mathrm{(\wedge R)}~\frac{\Gamma\Rightarrow A\quad \Gamma \Rightarrow B}{\Gamma \Rightarrow A\wedge B}
\end{displaymath}
\begin{displaymath}
(\mathrm{\vee L})~\frac{\Gamma[A]\Rightarrow C\quad \Gamma[B]\Rightarrow C}{\Gamma[A\vee B] \Rightarrow C}\quad
(\mathrm{\vee R)}~ \frac{\Gamma \Rightarrow A_i}{\Gamma\Rightarrow A_1\vee A_2}\quad(i=1,2)
\end{displaymath}
\begin{displaymath}
(\mathrm{\owedge C}) ~\frac{\Gamma[\Delta\owedge \Delta]\Rightarrow A}{\Gamma [\Delta]\Rightarrow A}\quad
(\mathrm{\odot C}) ~ \frac{\Gamma[(\Lambda\odot\Delta)\odot \Delta]\Rightarrow A}{\Gamma[\Lambda\odot\Delta]\Rightarrow A}~(\Lambda~\mathrm{is~not~empty})
\end{displaymath}
\begin{displaymath}
(\mathrm{\owedge E})~\frac{\Gamma[\Delta\owedge \Lambda]\Rightarrow A}{\Gamma[ \Lambda\owedge \Delta]\Rightarrow A}\quad
(\mathrm{Cut}) ~\frac{\Delta \Rightarrow A\quad \Gamma[A] \Rightarrow B}{\Gamma[\Delta] \Rightarrow B}
\end{displaymath}
\begin{displaymath}
(\mathrm{W}^1) ~\frac{\Gamma[\Delta]\Rightarrow A}{\Gamma[\Delta'*\Delta]\Rightarrow A}\quad(\mathrm{W}^2) ~\frac{\Gamma[\Delta]\Rightarrow A}{\Gamma[\Delta*\Delta']\Rightarrow A}\quad(*\in\{\owedge,\odot\})
\end{displaymath}
\begin{displaymath}
(\mathrm{\owedge A^1}) ~ \frac{\Gamma[(\Delta_1\owedge \Delta_2)\owedge \Delta_3] \Rightarrow A}{\Gamma[\Delta_1\owedge (\Delta_2\owedge \Delta_3)] \Rightarrow A}\quad (\mathrm{\owedge A^2}) ~ \frac{\Gamma[\Delta_1\owedge( \Delta_2\owedge \Delta_3)] \Rightarrow A}{\Gamma[(\Delta_1\owedge \Delta_2)\owedge \Delta_3] \Rightarrow A}
\end{displaymath}
The rules $(\imp\mathrm{R})$ and $(\leftarrow\mathrm{R})$ are restricted to nonempty sequences $\Gamma$.
The following semi-associativity $\odot$-rule is admissible in $\mathsf{L_{RBL}}$:
\begin{displaymath}
(\mathrm{\odot A^1}) ~ \frac{\Gamma[(\Delta_1\odot\Delta_2)\odot \Delta_3] \Rightarrow A}{\Gamma[\Delta_1\odot (\Delta_2\odot \Delta_3)] \Rightarrow A}
\end{displaymath}

By standard techniques (cf. \cite{Bus10}), it is easy to show that the sequent calculus $\mathsf{L_{RBL}}$ is equivalent to the algebraic system $\mathsf{S_{RBL}}$, i.e., for any sequent $\Gamma\Imp A$, $\vdash_{\mathsf{S_{RBL}}} \mu(\Gamma)\Imp A$ iff $\vdash_{\mathsf{L_{RBL}}}\Gamma\Imp A$.

For showing the subformula property of $\mathsf{L_{RBL}}$, we prove that the cut rule except the case that the cut formula is $\bot$ or $\top$ can be eliminated. We introduce the generalized mix rule instead of (Cut) and show the `mix elimination', i.e., every sequent derivable in $\mathsf{L_{RBL}}$ has a derivation without using the following mix rule:
\begin{displaymath}
(\mathrm{Mix})~~\frac{\Delta \Rightarrow A; \quad \Gamma[A]\ldots[A] \Rightarrow B}{\Gamma[\Delta]\ldots[\Delta] \Rightarrow B}
\end{displaymath}
where $\Gamma[A]\ldots[A]$ denotes the formula structure containing at least one occurrence of $A$, and $\Gamma[\Delta]\ldots[\Delta]$ denotes the formula structure obtained by replacing at least one occurrence of $A$ in $\Gamma[A]\ldots[A]$ by the formula structure $\Delta$. By $\mathsf{L^m_{RBL}}$, we mean the system obtained from $\mathsf{L_{RBL}}$ by replacing the $(\mathrm{Cut})$ by the above $(\mathrm{Mix})$ rule. Obviously, $\mathsf{L_{RBL}}$ is equivalent to $\mathsf{L^m_{RBL}}$.
\begin{theorem}\label{thm:mix}
A sequent is derivable in $\mathsf{L^m_{RBL}}$ has a derivation without using $(\mathrm{Mix})$ except the case that the mix formula is $\bot$ or $\top$.
\end{theorem}
\begin{proof}
We prove the admissibility of the rule $(\mathrm{Mix})$ by induction: 1) on the complexity of the mixed formula $A$; 2) on the length of the proof of $\Gamma[A]\ldots[A]\Imp B$; and 3) on the length of the proof of $\Delta\Imp A$. Let $\Delta\Imp A$ be obtained by rule $R_1$, and $\Gamma[A]\ldots[A]$ by rule $R_2$. We consider two cases.

Case 1. $A$ is not produced by $R_1$ or $R_2$. Let us consider the following subcases.

Case 1.1 $R_1=Id$ or $R_2=Id$. Assume $R_1=Id$. Then $\Delta = A$ and hence $\Gamma[\Delta]\ldots[\Delta] = \Gamma[A]\ldots[A]$. Thus we can eliminate this application of $(\mathrm{Mix})$. The case of $R_2=Id$ is similar.

Case 1.2 $R_1=(\bot)$. Note that the case $R_2=(\bot)$ does not apply to the mix rule. When $R_1=(\bot)$, the antecedent of the conclusion $\Gamma[\Delta]\ldots[\Delta]$ contains $\bot$. Hence by $(\bot)$ and $(\mathrm{W})$ we get the conclusion sequent.

Case 1.3 $R_2=(\top)$. Note that the case $R_1=(\top)$ does not occur. Let the premises of $(\mathrm{Mix})$ are $\Delta\Imp A$ and $A\Imp \top$, and the conclusion of $(\mathrm{Mix})$ be $\Delta\Imp \top$. Then $\Delta\Imp\top$ is obtained by $(\mathrm{W}^1)$ and axiom $(\top)$.

Case 1.4 $R_1=(\cdot\mathrm{L})$ or $R_2=(\cdot\mathrm{L})$. Assume $R_1=(\cdot\mathrm{L})$. Let the premise of $R_1$ be $\Delta[C\odot D]\Imp A$. We use mix rule to $\Delta[C\odot D]\Imp A$ and $\Gamma[A]\ldots[A]\Imp B$, and get $\Gamma[\Delta[C\odot D]]\ldots[\Delta[C\odot D]]\Imp B$. Then by the rule $(\cdot\mathrm{L})$ we get $\Gamma[\Delta[C\cdot D]]\ldots[\Delta[C\cdot D]]\Imp B$. The application of mix rule can be eliminated by induction 3). The case of $R_2=(\cdot\mathrm{L})$ is analogous.

Case1.5 $R_1=(\wedge\mathrm{L})$ or $R_2=(\wedge\mathrm{L})$. The proof is quite similar to case 1.4.

Case 1.6  $R_2=(\cdot\mathrm{R})$. Assume that the premises of $R_2$ are $\Gamma_1[A]\ldots[A]\Imp B_1$ and $\Gamma_2[A]\ldots[A]\Imp B_2$. We apply mix rule to $\Delta\Imp A$ and $\Gamma_1[A]\ldots[A]\Imp B_1$. We get $\Gamma_1[\Delta]\ldots[\Delta]\Imp B_1$. Similarly we get $\Gamma_2[\Delta]\ldots[\Delta]\Imp B_2$. Then by $(\cdot\mathrm{R})$, we get $\Gamma_1[\Delta]\ldots[\Delta]\owedge\Gamma_2[\Delta]\ldots[\Delta]\Imp B_1\cdot B_2$. The application of mix rule can be eliminated by induction 2).

Case 1.7 $R_2 = (\wedge\mathrm{R})$. The proof is quite similar to the case 1.6.

Case 1.8  $R_2 = (\vee\mathrm{R})$. The proof is quite similar to the case $R_2=(\cdot\mathrm{L})$.

Case 1.9 $R_1 = (\vee\mathrm{L})$ or $R_2 = (\vee\mathrm{L})$. For the case $R_1 = (\vee\mathrm{L})$,
assume the premises of $R_1$ are  $\Delta[C]\Imp A$ and $\Delta[D]\Imp A$. We apply mix rule to $\Delta[C]\Imp A$ and $\Gamma[A]\ldots[A]\Imp B$, and get $\Gamma[\Delta[C]]\ldots\Gamma[\Delta[C]]\Imp B$. Similar, we get $\Gamma[\Delta[D]]\ldots\Gamma[\Delta[D]]\Imp B$. Then by $(\vee\mathrm{L})$, $\Gamma[\Delta[C\vee D]]\ldots\Gamma[\Delta[C\vee D]]\Imp B$. The case $R_2 = (\vee\mathrm{L})$ is rather similar to the case $R_2=(\cdot\mathrm{R})$.

Case 1.10 $R_1=(\imp\mathrm{L})$ or $R_2=(\imp\mathrm{L})$. For the case $R_1=(\imp\mathrm{L})$, assume the premises of $R_1$ are $\Delta_1\Imp C$ and $\Delta_2[D]\Imp A$. First we apply mix rule to $\Delta_2[D]\Imp A$ and $\Gamma[A]\ldots[A]\Imp B$ and get $\Gamma[\Delta_2[D]]\ldots[\Delta_2[D]]\Imp B$. Then we apply $(\imp\mathrm{L})$ to $\Delta_1\Imp C$ and $\Gamma[\Delta_2[D]]\ldots[\Delta_2[D]]\Imp B$, and get the sequent
$\Gamma[\Delta_2[\Delta_1\odot (C\imp D)]]\ldots[\Delta_2[\Delta_1\odot (C\imp D)]]\Imp B$. By induction 2), the use of mix rule can be eliminated. For the case $R_2=(\imp\mathrm{L})$, let the conclusion be $\Gamma'[\Delta'[A]\ldots[A]\odot(C\imp D)][A]\ldots[A]\Imp B$, and the premises $\Delta'[A]\ldots[A]\Imp C$ and $\Gamma'[D][A]\ldots[A]\Imp B$. We apply mix rule first to $\Delta\Imp A$ and $\Delta'[A]\ldots[A]\Imp C$, $\Gamma'[D][A]\ldots[A]\Imp B$ respectively. Again, apply $(\imp\mathrm{L})$ to both resulting sequents. We obtain the sequent $\Gamma'[\Delta'[\Delta]\ldots[\Delta]\odot(C\imp D)][\Delta]\ldots[\Delta]\Imp B$. By induction 2), applications of mix rule can be eliminated.

Case 1.11 $R_1,R_2\in \{\owedge\mathrm{A}^1, \owedge\mathrm{A}^2, \owedge\mathrm{C}, \odot\mathrm{C}, \owedge\mathrm{E}, \mathrm{W}\}$. Apply the mix rule first to premises, and then use the corresponding rule. It is easy to eliminate applications of mix rule by induction 2).

Case 2. $A$ is created by $R_1$ and $R_2$. We consider the following subcases.

Case 2.1 $A=A_1\cdot A_2$. Let the two premises of mix rule are the following: $\Delta_1\odot\Delta_2\Imp A_1\cdot A_2$ which is obtained from $\Delta_1\Imp A_1$ and $\Delta_2\Imp A_2$ by $(\cdot\mathrm{R})$, and $\Gamma[A_1\cdot A_2][A]\ldots[A]\Imp B$ which is obtained from $\Gamma[A_1\odot A_2][A]\ldots[A]\Imp B$ by $(\cdot \mathrm{L})$.
Now we apply mix rule to $\Delta_1\odot\Delta_2\Imp A_1\cdot A_2$ and $\Gamma[A_1\odot A_2][A]\ldots[A]\Imp B$, and get $\Gamma[A_1\odot A_2][\Delta_1\odot\Delta_2]\ldots[\Delta_1\odot\Delta_2]\Imp B$. Then by applying mix rule to $\Delta_1\Imp A$ and the resulting sequent, we get $\Gamma[\Delta_1\odot A_2][\Delta_1\odot\Delta_2]\ldots[\Delta_1\odot\Delta_2]\Imp B$. Similary, we get $\Gamma[\Delta_1\odot \Delta_2][\Delta_1\odot\Delta_2]\ldots[\Delta_1\odot\Delta_2]\Imp B$. Thus by induction 2) and 1), mix rule can be eliminated.

Case 2.2 $A=A_1\imp A_2$ or $A_1\leftarrow A_2$. The proof is similar to case 2.1.

Case 2.3 $A=A_1\wedge A_2$. Let the two premises of mix rule are the following: $\Delta\Imp A_1\wedge A_2$ which is obtained from $\Delta\Imp A_1$ and $\Delta\Imp A_2$ by $(\wedge\mathrm{R})$, and $\Gamma[A_1\wedge A_2][A]\ldots[A]\Imp B$ which is obtained from $\Gamma[A_1\owedge A_2][A]\ldots[A]\Imp B$ by $(\wedge\mathrm{L})$.
Now first we apply mix rule to $\Delta\Imp A_1\wedge A_2$ and $\Gamma[A_1\owedge A_2][A]\ldots[A]\Imp B$, and get $\Gamma[A_1\owedge A_2][\Delta]\ldots[\Delta]\Imp B$. Then by applying mix rule to $\Delta\Imp A_1$ and $\Delta\Imp A_2$ and the resulting sequent, we get $\Gamma[\Delta\owedge\Delta][\Delta]\ldots[\Delta]\Imp B$. By $(\owedge\mathrm{C})$, we get $\Gamma[\Delta]\ldots[\Delta]\Imp B$.  By induction 2) and 1), mix rule can be eliminated.

Case 2.4 $A=A_1\vee A_2$. The proof is quite similar case 2.3.
\end{proof}

\begin{corollary}[Subformula Property]
If $\Gamma\Imp A$ has a derivation in $\mathsf{L_{RBL}}$, then all formulae in the derivation are $\top,\bot$, or subformulae of $\Gamma,A$.
\end{corollary}

\begin{theorem}[Disjunction Property]
For any $\mc{L}_{\mathrm{RBL}}$-formulae $A$ and $B$, if $\top\Imp A\vee B$ is derivable in $\mathsf{L_{RBL}}$, then $\top\Imp A$ or $\top\Imp B$ is derivable.
\end{theorem}
\begin{proof}
Assume $\vdash_{\mathsf{L_{RBL}}}\top\Imp A\vee B$. Then $\vdash_{\mathsf{L^m_{RBL}}}\top\Imp A\vee B$. By Theorem \ref{thm:mix}, the last rule can only be $(\vee \mathrm{R})$.
\end{proof}

Since $\mathsf{L_{RBL}}$ is a conservative extension of $\mathsf{BPL}$, we can conclude the following disjunction property of $\mathsf{BPL}$.

\begin{corollary}
For any $\mc{L}_{\mathrm{BPL}}$-formulae $A$ and $B$, if $\vdash_{\mathsf{BPL}} A\vee B$, then $\vdash_{\mathsf{BPL}} A$ or $\vdash_{\mathsf{BPL}} B$.
\end{corollary}

\section{Relationships with Other Logics}
The sequent calculus $\mathsf{LJ}$ for $\mathsf{Int}$ can be obtained from $\mathsf{L_{RBL}}$ 
by replacing $(\odot \mathrm{C})$ by the following full contraction rule and semi-associative rule:
\begin{displaymath}
(\mathrm{\odot C^*}) ~ \frac{\Gamma[\Delta\odot \Delta]\Rightarrow A}{\Gamma[\Delta]\Rightarrow A}
\quad
(\mathrm{\odot A^2}) ~ \frac{\Gamma[\Delta_1\odot(\Delta_2\odot \Delta_3)] \Rightarrow A}{\Gamma[(\Delta_1\odot \Delta_2)\odot \Delta_3] \Rightarrow A}
\end{displaymath}
and dropping the nonempty restriction on $(\imp\mathrm{R})$ and $(\leftarrow\mathrm{R})$.

By $(\odot C^*)$, the binary operator $\cdot$ is equivalent to $\wedge$, and by residuation, the implication becomes Heyting. Thus the following exchange rule in $\mathsf{LJ}$ is derivable from $\odot\mathrm{C}^*$ and weankening rules:
      \begin{displaymath}
(\mathrm{\odot E}) ~ \frac{\Gamma[\Delta\odot \Lambda]\Rightarrow A}{\Gamma[\Lambda\odot\Delta]\Rightarrow A}
   \end{displaymath}
but it is not derivable in $\mathsf{L_{RBL}}$.

The difference between $\mathsf{L_{RBL}}$ and $\mathsf{Int}$ can also be shown by the the following example. Let us consider the two formulae:
$(\dag)$ $p\wedge (p\imp q)\imp (\top\imp q)$; $(\ddag)$ $p\wedge (p\imp q)\imp q$.
The formula $(\ddag)$ is provable in $\mathsf{Int}$ but not in $\mathsf{BPL}$. In fact, $\mathsf{Int}=\mathsf{BPL}+(\ddag)$ (\cite{IK07}).
However, the formula $(\dag)$ is a theorem of $\mathsf{BPL}$ and hence provable in $\mathsf{L_{RBL}}$. The proof goes as follows. First we can derive the sequent for any $\mc{L}_{\mathrm{BPL}}$-formulae $A$ and $B$:
\begin{center}
($\#$) $\top\cdot(A\wedge B)\Imp (\top\cdot A)\cdot B$
\end{center}
in $\mathsf{L_{RBL}}$: from ($\mathrm{\odot C}$) and ($\mathrm{\cdot R}$), we get $\top\cdot (A\wedge B)\Imp (\top\cdot (A\wedge B))\cdot (A\wedge B)$. By $\mathrm{(\wedge L)}$ and ($\mathrm{\cdot R})$, we obtain $(\top\cdot (A\wedge B))\cdot (A\wedge B))\Imp  (\top\cdot A)\cdot B$. Then by (Cut) we get the required sequent. Now we can derive $\top\Imp p\wedge(p\imp q)\imp (\top\imp q)$ in $\mathsf{L_{RBL}}$: Since $(\top\cdot p)\cdot (p\imp q) \Imp q$ is provable, by $(\mathrm{Cut})$ and ($\#$), we get $\top\odot (p\wedge (p\imp q) )\Imp q$. Then by $(\imp\mathrm{R})$ we get the required sequent.

The contraction rule $(\odot\mathrm{C})$ used in above derivation is significant. It is different from $(\odot C^*)$ in $\mathsf{LJ}$ since a nonempty prefix is needed for contraction. This point is also reflected in the difference between the two formulae $(\dag)$ and $(\ddag)$. In the consequent of $(\dag)$, we use $\top$ as a prefix so that $(\dag)$ is provable in $\mathsf{L_{RBL}}$. If we add $\top\imp A\Imp A$ as an additional axiom to $\mathsf{L_{RBL}}$, we can derive $p\wedge(p\imp q)\Imp q$.

Ishii et.al. \cite{IKK01} introduced the sequent system $\mathsf{LBP}$ for $\mathsf{BPL}$, and proved that for any $\mc{L}_{\mathrm{BPL}}$-formula $A$,
$\vdash_{\mathsf{LBP}~}\Imp A$ iff $\vdash_{\mathsf{BPL}} A$. By the conservative extension theorem \ref{thm:con} and, we get the following:
for any $\mc{L}_{\mathrm{BPL}}$-formula $A$, $\vdash_{\mathsf{BPL}} A$ iff $\vdash_{\mathsf{LBP}~}\Imp A$ iff $\vdash_{\mathsf{L_{RBL}}}\top\Leftrightarrow A$. Now let us show one more result on the sequent calculus $\mathsf{LBP}$ through our conservative extension theorems. The following rule in $\mathsf{LBP}$ is used:
\begin{displaymath}
(\mathrm{\imp_{n=0}})~\frac{\Sigma,A\Rightarrow B}{\Sigma \Rightarrow A\imp B}
 \end{displaymath}
\begin{proposition}\label{prop:seq}
For any $\mc{L}_{\mathrm{BPL}}$-formulae $A$ and $B$, $\vdash_{\mathsf{LBP}} B\Imp A$ iff $\vdash_{\mathsf{LBP}}\Imp B\imp A$.
\end{proposition}
\begin{proof}
We sketch the proof. The left-to-right direction is obvious by the rule $(\imp_{n=0})$ in $\mathsf{LBP}$. Conversely, assume $\not\vdash_{\mathsf{LBP}} B\Imp A$. By completeness of $\mathsf{LBP}$ (\cite{IKK01}), there exists a BPL-model $\mf{M}$ such that $\mf{M}\not\models B\Imp A$. Thus there exists a state $x$ in $\mf{M}$ such that $\mf{M},x\models B$ but $\mf{M},x\not\models A$. Assume that $\mf{M}$ is generated by $x$ (otherwise, we consider the submodel of $\mf{M}$ generated from $x$, cf.\cite{CZ97}). Let $\mf{M}'$ be a new BPL-model obtained from $\mf{M}$ by adding a new state $x'$ which is accessible to $x$ such that the valuation is the same as that in $\mf{M}$. Then by induction on the complexity of $\mc{L}_{\mathrm{BPL}}$-formulae, it is easy to show that $(\mf{M},x)$ and $(\mf{M}',x)$ are $\mc{L}_{\mathrm{BPL}}$-equivalent. Hence $\mf{M}',x'\not\models B\Imp A$. Thus $\mf{J}^{\mf{M}'}x_i\not\models B\Imp A$. By construction of $\mf{J}^{\mf{M}'}$, it is easy to see that $\mf{J}^{\mf{M}'}\not\models B\cdot\top\Imp A$. By lemma \ref{lemma:arbasound}, $\not\vdash_{\mathsf{S^*_{RBL}}}B\cdot \top\Imp A$. Hence $\not\vdash_{\mathsf{S^*_{RBL}}}\top\Imp B\imp A$.
By conservative extension theorem \ref{thm:con}, $\not\vdash_{\mathsf{BPL}} B\imp A$.
Therefore, $\not\vdash_{\mathsf{LBP}} \Imp B\imp A$.
\end{proof}

\begin{remark}
In \cite{IKK01}, Ishii et.al. proved only that the proposition \ref{prop:seq} holds in the sequent calculus $\mathsf{LFP}$ for Visser's formal provability logic. It does not hold generally for any sequent $\Gamma\Imp A$, since the structural operation comma in sequent systems $\mathsf{LBP}$ and $\mathsf{LFP}$ is interpreted by $\wedge$, while the structural operation $\odot$ means $\cdot$ in $\mathsf{L_{RBL}}$.
\end{remark}

Finally, we consider the implicational fragments.
By an implicational formula $A$ we mean a formula constructing from propositional letters using only $\imp$.
The implicational fragment $\mathsf{BPL}^\imp$ is the set of all implicational formulas provable in $\mathsf{BPL}$. Kikuchi presented in \cite{Kik01} a Hilbert-style axiomatization of $\mathsf{BPL}^\imp$ which consists of the following axioms and inference rule:
\begin{itemize}
\item[](I) $A\imp A$
\item[](K) $A\imp (B\imp A)$
\item[](B$^*$) $(\Gamma\imp (B\imp C))\imp ((\Gamma\imp (A\imp B))\imp (\Gamma\imp (A\imp C)))$
\item[](MP) from $A$ and $A\imp B$ infer $B$.
\end{itemize}
where $\Gamma$ is a finite sequence of formulas, and $\Gamma\imp A$ is defined as follows: (i)
$\Gamma\imp A = A$ if $\Gamma$ is empty; (ii) $B,\Gamma\imp A = B\imp (\Gamma\imp A)$.
The fragment $\mathsf{BPL}^\imp$ is complete with respect to the class of all transitive Kripke models.

The $\imp$-fragment of $\mathsf{RBL}$ is defined as the set $\mathsf{RBL}^\imp$ of all implicational formulae which are valid in $\mathbb{RBA}$.
Then by theorem \ref{thm:rblcon} on conservative extension, it is easy to see that $\mathsf{RBL}^\imp$ is a conservative extension of $\mathsf{BPL}^\imp$.

\section{Conclusion}
The basic implication in Visser's basic propositional logic $\mathsf{BPL}$ can be formalized
in substructural logic as the right residual of the product $\cdot$ (fusion) operation. The resulting
logic is the residuated basic logic $\mathsf{RBL}$ of the variety of residuated basic algebras.
We develop the algebraic system $\mathsf{S_{RBL}}$ which is a formalization of the equational logic of residuated basic algebras. We show that $\mathsf{S_{RBL}}$ is a conservative extension of $\mathsf{BPL}$. Consequently, $\mathsf{RBL}$ is a conservative extension of $\mathsf{BPL}$. The implicational fragment of $\mathsf{BPL}$ is equal to the $\imp$-fragment of $\mathsf{RBL}$. Moreover, we develop the Gentzen-style sequent formalization $\mathsf{L_{RBL}}$ for $\mathsf{S_{RBL}}$, and show that the cut elimination and subformula property hold for $\mathsf{L_{RBL}}$.

Finally, the interpolation lemma, finite model property, and decidability of our sequent calculus $\mathsf{L_{RBL}}$ can be proved by Buszkowski's method \cite{Bus11} for showing interpolation and finite embedding property for classes of residuated algebras. The proof will be presented in a forthcoming paper.

\bibliographystyle{plain}
\bibliography{rbl}

\paragraph{Acknowledgements.}
The authors were supported by the project of China National Social Sciences Fund (Grant no. 12CZX054).

\end{document}